\documentclass[11pt,a4paper]{article}
\usepackage{fancyhdr}
\usepackage[british]{babel}

\usepackage{amssymb, amsmath, amsthm, amsfonts, enumerate,bm}
\usepackage{amsmath,setspace,scalefnt}
\usepackage[usenames,dvipsnames,svgnames,table]{xcolor}
\usepackage{graphicx,tikz,caption,subcaption}
\usetikzlibrary{patterns}
\usetikzlibrary{shapes}
\usepackage{fullpage}
\usepackage{hyperref}
\usepackage{enumitem,verbatim}
\usepackage{multicol}
\usepackage{float}
\usepackage{cleveref}

\usepackage[margin=2.5cm]{geometry}

\definecolor{gray}{rgb}{0.25, 0.25, 0.25}

\usepackage{booktabs}  

\counterwithin{figure}{section}

\newtheorem{theorem}{Theorem}[section]
\newtheorem{lemma}[theorem]{Lemma}
\newtheorem{claim}[theorem]{Claim}
\newtheorem{cor}[theorem]{Corollary}

\newtheorem{conj}[theorem]{Conjecture}

\newtheorem*{observation*}{Observation}
\newtheorem{proposition}[theorem]{Proposition}
\newtheorem{problem}[theorem]{Problem}

\newtheorem*{question*}{Question}
\newtheorem{innerdefinition}{Definition}
\newenvironment{definition*}
  {
   \innerdefinition}
  {\endinnerdefinition}

\theoremstyle{definition}

\newtheorem{example}[theorem]{Example}

\theoremstyle{remark}

\newenvironment{poc}{\begin{proof}[Proof of claim]}{\end{proof}}

\newcommand{\C}[1]{{\protect\mathcal{#1}}}

\newcommand{\ceil}[1]{\lceil #1\rceil}
\newcommand{\floor}[1]{\lfloor #1\rfloor}

\usepackage{todonotes}

\newcommand{\EE}{\mathbb E}

\newcommand*{\abs}[1]{\lvert#1\rvert}

\title{More on Nosal's spectral theorem: Books and $4$-cycles}
\author{
Yongtao Li\thanks{Yau Mathematical Sciences Center, Tsinghua University, Beijing, China. Email: \url{yongtao_li@mail.tsinghua.edu.cn}. Supported by the Shuimu Scholar Program at Tsinghua University.}
\and
Hong Liu\thanks{Extremal Combinatorics and Probability Group (ECOPRO), Institute for Basic Science (IBS), Daejeon, South Korea. Email: \url{hongliu@ibs.re.kr}. Supported by the Institute for Basic Science (IBS-R029-C4).}
\and
Shengtong Zhang\thanks{Department of Mathematics, Stanford University, CA, USA. Email: \url{stzh1555@stanford.edu}. Supported by the Craig Franklin Fellowship in Mathematics at Stanford University.}
}
\date{\today}

\begin{document}
\maketitle

\begin{abstract}
Spectral graph theory studies how the eigenvalues of a graph relate to the structural properties of a graph. In this paper, we solve three open problems in spectral extremal graph theory which generalize the classical Tur\'{a}n-type supersaturation results.
   \begin{itemize}
    \item We prove that every $m$-edge graph $G$ with the spectral radius $\lambda (G) > \sqrt{m}$ contains at least $\frac{1}{144} \sqrt{m}$ triangles sharing a common edge. This result confirms a conjecture of Nikiforov, and Li and Peng. Moreover, the bound is optimal up to a constant factor.

    \item Next, for $m$-edge graph $G$ with $\lambda (G) > \sqrt{(1-\frac{1}{r})2m}$, we show that it must contain 
    $\Omega_r (\sqrt{m})$ copies of $K_{r+1}$ sharing $r$ common vertices. This confirms a conjecture of Li, Liu and Feng and unifies a series of spectral extremal results on books and cliques. Moreover, we also show that such a graph $G$ contains $\Omega_r  (m^{\frac{r-1}{2}})$ copies of $K_{r+1}$. This extends a result of Ning and Zhai for counting triangles. 

    \item We prove that every $m$-edge graph $G$ with $\lambda (G) > \sqrt{m}$ contains at least $(\frac{1}{8}-o(1)) m^2$ copies of 4-cycles, and we provide two constructions showing that the constant $\frac{1}{8}$ is the best possible. This result settles a problem raised by Ning and Zhai, and it gives the first asymptotics for counting degenerate bipartite graphs. 
   \end{itemize}
   The key to our proof are two structural results we obtain for graphs with large spectral radii on their maximum degree and on existence of large structured subgraphs, which we believe to be of independent interest.
\end{abstract}

\begin{center} \emph{In memory of Prof. Vladimir Nikiforov.} \end{center}



\section{Introduction} 
We study spectral problems for supersaturation phenomena in extremal graph theory. Spectral extremal graph theory has enjoyed tremendous growth in the past few decades thanks to its connections and applications 
to numerous other fields. Throughout the paper, for a graph $G$, we write $n$ and $m$ for the number of vertices and edges in $G$, respectively.

Our starting point is a classical result of Nosal \cite{Nos1970} (see, e.g., \cite{Niki2002}), which states
that if $G$ is triangle-free, then its spectral radius satisfies $\lambda(G) \le \sqrt{m}.$
Nosal's theorem strengthens the fundamental result of Mantel (see, e.g., \cite{Bol1978}) that 
if $G$ is triangle-free, then $m\le \lfloor n^2/4\rfloor$, 
with equality if and only if $G=T_{n,2}$, where $T_{n,r}$ is the $n$-vertex complete $r$-partite Tur\'an graph whose part sizes are as equal as possible. Indeed, using the Rayleigh formula, 
we have $ 2m/n \le \lambda (G) \le \sqrt{m}$, which gives  
$m \le \lfloor n^2/4\rfloor = e(T_{n,2})$ and 
$\lambda (G)\le \sqrt{\lfloor n^2/4\rfloor} = \lambda (T_{n,2})$. 
In honor of this foundational result, in this paper we say that an $m$-edge graph $G$ is \textbf{Nosal} if  it satisfies $\lambda(G) > \sqrt{m}.$
There are various generalizations of Nosal's theorem for other forbidden subgraph $F$ in the literature; see \cite{Niki2002,BN2007,LNW2021,Zhang2024} for cliques, 
\cite{BG2009,Niki2010laa} for complete bipartite graphs,
 \cite{CDT2024,LZS2024,LN2023,Zhang-wq-2024} for cycles, 
 \cite{CDT2023,FLSZ2024} for trees, 
 \cite{CFTZ20, ZLX2022,Niki2009cpc,LLP2024-AAM,LFP2024-triangular} for others. 

Supersaturation, appearing in many different problems in extremal combinatorics, refers to the phenomenon that when going beyond the extremal threshold, not just a single copy but a great number of forbidden structures would emerge. For instance, 
Rademacher (unpublished, see Erd\H{o}s \cite{Erd1955,Erdos1964}) showed that every 
$n$-vertex graph with at least $\lfloor n^2/4\rfloor +1$ edges contains at least $\lfloor n/2\rfloor$ triangles. Such problems have attracted a great deal of attentions; see, e.g.,~\cite{LS1975,LS1983,Mubayi2010,PY2017,HMY2021,LPS2020} for recent developments on  supersaturation. 

The study of spectral problems for supersaturation was initiated by
Bollob\'{a}s and Nikiforov \cite{BN2007} in 2007. In particular, they proved that the number of triangles $t(G)$ satisfies 
$t(G)\ge \frac{n^2}{12}(\lambda(G) - \frac{n}{2})$ and 
$t(G) \ge \frac{1}{3}\lambda(G) \bigl(\lambda^2(G) - m\bigr).$ 
The second inequality 
was independently proved by Cioab\u{a},  Feng,
Tait and Zhang \cite{CFTZ20}. 
Furthermore, Ning and Zhai \cite{NZ2021} showed that 
the equality holds if and only if 
$G$ is a complete bipartite graph. 
Moreover, Ning and Zhai also proved that if $G$ is Nosal with $m$ edges, 
 then $t(G)\ge \lfloor \frac{1}{2}(\sqrt{m}-1) \rfloor$; this bound is the best possible. 
The spectral supersaturation problems were also  
studied under the  condition $\lambda (G) > \lambda (T_{n,2})$ for $n$-vertex graphs $G$; 
see, e.g., \cite{NZ2021,ZL2022jgt,LFP2024-triangular,LFP-count-bowtie} for recent progresses. 

In this paper, we investigate the supersaturation for Nosal graphs.

\subsection{Books in Nosal graphs}
A \emph{book} of size $k$, denoted $B_k$, is a graph consisting of $k$ triangles sharing a common edge. For a graph $G$, denote by $bk(G)$ the size of the largest book it contains.
Erd\H{o}s \cite{Erd1962a} proved that every graph on $n$ vertices with $\lfloor n^2/4\rfloor +1$ edges satisfies $bk(G) = \Theta(n)$ and conjectured that $bk(G)> n / 6$, which was proven by Edwards (unpublished, see \cite[Lemma 4]{EFR1992}) 
and independently 
by Khad\v{z}iivanov and Nikiforov  \cite{KN1979}. 
Two alternative proofs were provided by 
Bollob\'{a}s and Nikiforov \cite{BN2005} and Li, Feng and Peng \cite[Sec. 4.3]{LFP2024-triangular}.  Problems on books have attracted considerable attention; see, e.g., \cite{FL2012,CFS2020-Siam,CFW2022,ZL2022jgt} and references therein.

Recently, Zhai, Lin and Shu \cite{ZLS2021} proved that if $G$ is Nosal with $m$ edges, then $G$ contains a copy of $K_{2,r}$ with $r >\frac{1}{4} \sqrt{m}$. Note that $B_{r}$ is $K_{2, r}$ with an extra edge. Zhai, Lin and Shu~\cite{ZLS2021} conjectured that $G$ also contains a large book. This was confirmed by Nikiforov \cite{Niki2021}, who proved that every $m$-edge Nosal graph $G$ satisfies $bk(G)>\frac{1}{12}\sqrt[4]{m}$. In the same paper, Nikiforov remarked that \emph{``The bound on $bk(G)$ seems far from optimal, as the multiplicative
constant and perhaps the exponent $1/4$ can be improved.''} Supporting Nikiforov's speculation, Li and Peng \cite{LP2022oddcycle} conjectured that the exponent can be improved to $1/2$. 

\begin{conj}[Nikiforov \cite{Niki2021}, Li--Peng \cite{LP2022oddcycle}] 
\label{conj-ZLS}
For every $m$-edge Nosal graph $G$, we have $$bk(G)=\Omega(\sqrt{m}).$$
\end{conj}

Our first result confirms this conjecture and shows that the order $\sqrt{m}$ is the best possible. 

\begin{theorem} \label{thm-booksize}
If $G$ is an $m$-edge Nosal graph, then $bk (G) > \frac{1}{144} \sqrt{m}$.
Furthermore, there exist $m$-edge Nosal graphs with no book of size larger than $(\frac{1}{3}+o(1))\sqrt{m}$.
\end{theorem}

In 1941, Tur\'{a}n \cite{Bol1978} extended Mantel's theorem, showing that if $G$ is a $K_{r+1}$-free graph on $n$ vertices, then $e(G)\le (1-\frac{1}{r})\frac{n^2}{2}$, with equality if and only if $r$ divides $n$ and $G=T_{n,r}$. In 1986, Wilf  \cite{Wil1986} proved a spectral extension that if $G$ is a $K_{r+1}$-free graph on $n$ vertices, then $\lambda (G) \le \left(1-\frac{1}{r }\right)n$, and equality holds if and only if $r$ divides $n$ and $G=T_{n,r}$. In 2002, Nikiforov \cite{Niki2002} showed a further extension that 
if $G$ is a $K_{r+1}$-free graph with $m$ edges,  then $\lambda^2 (G) \le 
 \left( 1-\frac{1}{r}\right)2m$, and
the extremal graphs were later characterized in \cite{Niki2006}; they are complete bipartite graphs for $r=2$, 
or regular complete $r$-partite graphs for $r\ge 3$. 

The \emph{generalized book} $B_{r,k}=K_r \vee I_k$ is a graph obtained  
 by joining every vertex of a clique $K_r$
to every vertex of an independent set $I_k$ of size $k$. Naturally, one wonders whether Conjecture \ref{conj-ZLS} can be extended to generalized books. This problem was proposed in \cite[Conjecture 1.20]{LLF2022}.

\begin{conj}[Li--Liu--Feng \cite{LLF2022}] \label{conj-B-rk}
Let $r\ge 2, k\ge 1$ be fixed and $m$ be large enough.  
If $G$ is a $B_{r,k}$-free graph with $m$ edges, 
then  $\lambda^2 (G)\le \left( 1-\frac{1}{r} \right) 2m.$ 
\end{conj}

Conjecture \ref{conj-B-rk} gives a unified 
extension on the spectral extremal results for triangles, books and cliques. 
Our next result resolves Conjecture \ref{conj-B-rk} in a strong sense, determining the correct order of largest generalized books guaranteed in graphs with large spectral radii.

\begin{theorem}
    \label{cor:joints}
    Every $m$-edge graph $G$ with $\lambda^2(G) > \left(1 - \frac{1}{r} \right)  2m$ contains a copy of $B_{r, k}$ of size $k = \Omega_r(\sqrt{m})$. 
    Furthermore, there are such graphs with largest generalized booksize $O_r(\sqrt{m})$.
\end{theorem}

As an application of Theorem \ref{cor:joints}, we can easily obtain that if $G$ is an $n$-vertex graph with the following condition: either $m> (1-\frac{1}{r})\frac{n^2}{2}$ edges or 
$\lambda (G)> (1-\frac{1}{r})n$, 
then $G$ contains a copy of $B_{r,k}$ 
with $k=\Omega_r(n)$. This can be regarded as an extension of a result of Zhai and Lin \cite{ZL2022jgt}, in which they proved that if $\lambda (G) > \lambda (T_{n,2})$, then $G$ contains a book $B_{2,k}$ with $k=\Omega(n)$. 

\subsection{Quadrilaterals and structures of Nosal graphs} 
We now turn our attention to supersaturations for bipartite graphs. Note that the star $K_{1,m}$ satisfies $\lambda(K_{1,m})=\sqrt{m}$ and does not contain any $4$-cycle.
A result of Nikiforov \cite{Niki2009} states that if $G$ is a Nosal graph with $m\ge 10$ edges, then $G$ contains a $4$-cycle; see, e.g., \cite{ZLS2021,ZS2022dm}. 
 Recently, Ning and Zhai \cite{NZ2021b} proved a corresponding supersaturation result, showing that if $G$ is a Nosal graph with $m\ge 3.6\times 10^9$ edges, then 
$G$ has at least $m^2/2000$ copies of $C_4$. Let 
\[ f(m)= \min\{\text{$\#C_4$ in $G$}: e(G)=m \text{~and~} 
\lambda (G) > \sqrt{m}\}\] 
be the number of $C_4$ guaranteed in every $m$-edge Nosal graph. The result of Ning and Zhai \cite{NZ2021b}  can be written as 
$f(m) > m^2/2000.$
Ning and Zhai proposed the following problem.

\begin{problem}[Ning--Zhai \cite{NZ2021b}]
    Determine the limit $\lim\limits_{m\to \infty} \frac{f(m)}{m^2}$. 
\end{problem}

We point out that  
it is possible to improve the constant $1/2000$ by a similar argument as in \cite{NZ2021b}. 
However, it appears difficult to determine the limit exactly using Ning--Zhai's approach. 
We resolve 
this problem using a different approach. 

\begin{theorem} \label{thm:lim}
Every $m$-edge Nosal graph has at least $( \frac{1}{8} - o(1)) m^2$ copies of $C_4$. Furthermore, the constant $\frac{1}{8}$ is optimal. In other words, we have $\lim\limits_{m\to \infty} \frac{f(m)}{m^2} 
= \frac{1}{8}.$ 
\end{theorem}

The key ingredients and bulk of the proof of~Theorem \ref{thm:lim} are two structural results we obtain for Nosal graphs, which we believe to be of independent interest and may have further applications.

The first one bounds the maximum degree of a Nosal graph.
\begin{theorem}
        \label{lem:Nosal-max-degree}
    If $G$ is an $m$-edge Nosal graph, 
    then for sufficiently large $m$,  
    \[ \Delta(G) \leq \frac{m}{2} + m^{0.99}.\] 
\end{theorem}

One difficulty in proving Theorem \ref{lem:Nosal-max-degree} comes from the fact that  $K_{1, m}$ has spectral radius exactly equal to $\sqrt{m}$, so one needs to show that tweaking $K_{1, m}$ can only decrease the spectral radius. Another difficulty lies in the fact that there are two very different looking graphs (see Examples \ref{exam-C4-1} and \ref{exam-C4-2}) which are both extremal graphs up to a lower order term. We remark that the lower order term $m^{0.99}$ is far from optimal, but it is sufficient for our purpose.

The second one is a structural dichotomy for Nosal graphs. It says that any Nosal graph has a large subgraph that is either bipartite or has maximum degree $o(m)$. For instance, in Example \ref{exam-C4-1} the clique has maximum degree $o(m)$, and the graph in Example \ref{exam-C4-2} can be made bipartite by removing a single edge. 

\begin{theorem}
    \label{lem:Nosal-structure}
    For any $\varepsilon > 0$, there exists a constant $N(\varepsilon)$ such that the following holds. If $G$ is an $m$-edge Nosal graph, then there exists a subgraph $G'$ of $G$ with 
    \[ \lambda(G') > (1 - \varepsilon) \sqrt{m} - N(\varepsilon) \]
    such that one of the following holds: 
    \begin{itemize}
        \item[\rm (a)] $G'$ is bipartite; 
        \item[\rm (b)] $G'$ has maximum degree at most $\varepsilon m$.
    \end{itemize}
\end{theorem}

\subsection{Applications on phase transitions and our approach}

We would like to point out an interesting analogy between Nosal graphs and $n$-vertex graphs with more than $ n^2/4$ edges. An avenue of new directions that fit into our framework is to find results about certain graph parameters that exhibit a phase transition/jump when $e(G)$ changes from $\lfloor n^2/4 \rfloor$ to 
$\lfloor n^2/4 \rfloor +1$ and extend such results to Nosal graphs. Some known examples are: (i) the number of triangles jumps from $0$ to 
$\Omega (n)$, see \cite{Erd1955,Erdos1964}; (ii) the booksize jumps from $0$ to $\Omega (n)$, see \cite{KN1979,BN2005,LFP2024-triangular}; (iii) the number of triangular edges jumps from 
 $0$ to $\Omega (n)$, see \cite{EFR1992};  (iv) the size of largest chordal subgraph 
 jumps from $0$ to $\Omega (n)$, see \cite{EGOZ1989}; (v) the number of copies of $C_4^+$ jumps from 
$0$ to $\Omega (n^2)$, see \cite{Mubayi2010}; (vi) the number of $K_4$-saturating-edges jumps from $0$ to $\Omega (n^2)$, see \cite{BL2014}. 
As corollaries of our results, we obtain spectral analogs of the jump phenomena of these parameters for Nosal graphs. We defer detailed discussions to Section~\ref{sec-5}.

Lastly, we remark that results obtained under the edge-spectral condition $\lambda(G)>\sqrt{m}$ of being Nosal is a generalization of both the density condition $e(G)>\lfloor n^2/4 \rfloor$ and the vertex-spectral condition $\lambda(G)>\sqrt{\lfloor n^2/4 \rfloor}$. Indeed, as mentioned before, by the Rayleigh formula graphs satisfying one of the latter two conditions must be Nosal. What is more important is that the results under vertex-spectral condition apply only to dense graphs with quadratically many edges, whereas the edge-spectral extensions can apply to graphs of all density (see, e.g., the sparse graphs in Examples \ref{exam-C4-1} and \ref{exam-C4-2}). It would be interesting to obtain more edge-spectral extensions of extremal results; see Theorem~\ref{thm-jsr+1} for one such example for joints (cf.~\Cref{cor:jsr}).

\paragraph{Our approach.} For Theorem \ref{thm-booksize}, we first prove a weighted version of 
Edwards--Khad\v{z}iivanov--Nikiforov's result (Lemma \ref{lem:Edwards-Blowup}) on the appearance of large books in dense graphs. Our main tool for this step is a novel random blowup argument. We then introduce a weighting on the Nosal graph $G$ where the vertices are weighted according to the Perron--Frobenius eigenvector $\bm{x}$. Assigning weights to edges turns out to be more delicate. If we assign the natural uniform weights to edges of $G$, the spectral radius condition implies that this weighted $G$ is dense and hence we can find a large weighted book. However, we cannot pull back this weighted book to an actual (unweighted) one in $G$ due to lack of control on the $\ell_{\infty}$-norm of $\bm{x}$. The key step of our proof is to carefully design a set of edge weights to circumvent this issue. To prove the more general~Theorem \ref{cor:joints}, we first extend the argument in Theorem \ref{thm-booksize} to find a large joint (Theorem \ref{thm-jsr+1}), which is a collection of cliques sharing a common edge. We then obtain a desired large generalized book from a large joint using the Kruskal--Katona theorem. 

For Theorem \ref{thm:lim}, using the fact that the fourth moment of the eigenvalues of $G$ counts the number of $4$-walks in $G$, one can estimate the number of $4$-cycles in $G$ with a combination of spectral and degree information (Lemma \ref{lem:LL-counting}). Unfortunately, this estimate alone is not strong enough to guarantee even a single copy of $C_4$ for Nosal graphs (see~\cref{eq:trace-not-enough}). This is where the two structural results (Theorems \ref{lem:Nosal-max-degree} and \ref{lem:Nosal-structure}) kick in. We lower bound the number of $C_4$ in the large subgraph $G'$ from~Theorem \ref{lem:Nosal-structure} instead. If $G'$ is bipartite, we gain a factor of 2 improvements on both the fourth moment of the eigenvalues and the $\ell_2$-norm of the degree sequence, and the desired $C_4$ count follows. If $G'$ has maximum degree $o(m)$, we can then bound the  $\ell_2$-norm of the degree sequence by $o(m^2)$, which is also sufficient to finish the proof.

\paragraph{Organization.} 
We prove~\Cref{thm-booksize} in Section \ref{sec-2}. Then in Section \ref{sec-3}, we extend our technique to prove~\Cref{cor:joints}. The proof of~\Cref{thm:lim,lem:Nosal-max-degree,lem:Nosal-structure} are given in Section \ref{sec-4}. In Section \ref{sec-5}, 
we provide some  applications of our results on phase transitions in Nosal graphs. In Section \ref{sec-6}, we conclude with some related spectral graph problems.

\paragraph{Notation.}  
All graphs in this paper are simple and undirected. Throughout, $G$ is a graph with vertex set $V$ and edge set $E$. Unless indicated otherwise, $n$ denotes the number of vertices of $G$, and $m$ denotes the number of edges of $G$. The degree of a vertex $i\in V$ is denoted by
$d_i$. The maximum degree of $G$ is denoted by $\Delta(G)$. The set of neighbors of a vertex $i\in V$ is denoted by $N(i)$. 
The adjacency matrix $A(G)$ of a graph $G$ 
is defined as a $V\times V$ matrix 
with $a_{i,j}=a_{j,i}=1$ if and only if $\{i,j\}\in E(G)$, 
and $a_{i,j}=a_{j,i}=0$ otherwise. 
Since $A(G)$ is real and symmetric, all eigenvalues of
$A(G)$ are real and can be sorted as
$\lambda_1\ge \lambda_2 \ge \cdots \ge \lambda_n$. 
Let $\lambda (G)$ be the largest eigenvalues of  $G$, which is 
 known as the {\it spectral radius} of  $G$. 
 By the Perron--Frobenius theorem, 
we have $\lambda(G) \geq \abs{\lambda_i}$ for any $i \in [n]$, and there exists a nonnegative unit eigenvector (called the Perron--Frobenius eigenvector) $\bm{x} = (x_i)_{i \in V}$
corresponding to $\lambda (G)$. 
In particular, for each $i \in V$ we have $\lambda (G) x_i = 
\sum_{j=1}^n a_{i,j} x_j =  \sum_{j\in N(i)} x_j.$
In the sequel, we shall write $\sum_{\{i,j\} \in E}$ for the sum 
over each edge in $E$ once.

\section{Proof of Theorem \ref{thm-booksize}}
\label{sec-2}

We first give three different constructions to show that the order of magnitude $\sqrt{m}$ in Theorem \ref{thm-booksize} is the best possible. Without loss of generality, we assume that $m$ is a perfect square.

\begin{example}\label{exam-booook}
Let $ s = \sqrt{m} +1$ and $t=m- {s \choose 2} \approx m/2$. 
We choose $H$ as any triange-free graph with 
$t$ edges. Define $K_s\circ H$ as the graph obtained 
from $K_s$ and $H$ by identifying a vertex. We can see that $\lambda(K_s \circ H) > \lambda (K_s) = \sqrt{m}$ 
and the booksize $bk(K_s\circ H)= \sqrt{m} -1$. 
\end{example}

\begin{example} \label{exam-KST}
We define $H=K_{s,t}^+$ as the graph obtained from 
the complete bipartite graph $K_{s,t}$ by adding an edge to the part of size $s$. Then the number of edges in $H$ is $m=st+1$. 
One can compute 
that $\lambda (H) > \sqrt{m}$ whenever $s< 4(t+1)$. 
Setting $s=2\sqrt{m}+1$ and $t=\frac{m-1}{2\sqrt{m}+1}$, 
we have $\lambda (H) > \sqrt{m}$ and $H$ has booksize $bk(H) = t \approx \frac{1}{2}\sqrt{m}$.
\end{example}

\begin{example} \label{exam-prism}
Let $C_3^{\square}$ be the triangular prism consisting of two disjoint triangles and a perfect matching joining them.  
We define $G$ as
 the blow-up of $C_3^{\square}$, where we 
replace each vertex in the upper triangle of $C_3^{\square}$ with an independent set of size $(k + 1)$ and each vertex in the lower triangle with an independent set of size $(k - 1)$. Each edge of $C_3^{\square}$ is replaced with a complete bipartite graph. 
Then $n=|G|=6k$, $m=e(G)=9k^2+3$ and $bk(G)=k+1 \leq \frac{1}{3} \sqrt{m} + 1$. The graph is Nosal since $m > n^2 / 4$ implies $\lambda (G)> \sqrt{m}$.
\end{example}

Our starting point for finding a large book is the following theorem of Edwards--Khad\v{z}iivanov--Nikiforov. 
We refer the reader to \cite{EFR1992,BN2005,LFP2024-triangular} for detailed proofs. 

\begin{lemma}[See \cite{KN1979}] \label{lem-Edw}
    If an $n$-vertex graph $G$ has more than $n^2 / 4$ edges, 
    then $bk(G) > {n}/{6}$.
\end{lemma}

A key ingredient in our proof is a weighted version of Lemma \ref{lem-Edw}. We deduce this result by a random blowup argument. We need Hoeffding's inequality as stated below.

\begin{lemma} \label{lem:Cher}
    Let $X$ be a sum of $K$ independent Bernoulli random variables. Then we have
   $$\mathrm{Pr} \left(|X-\mathbb{E}(X)| \ge K^{3/4} \right) \leq \exp\left(-\Omega(\sqrt{K}) \right).$$
\end{lemma}

Our weighted version of Lemma \ref{lem-Edw} reads as follows.
\begin{lemma}
    \label{lem:Edwards-Blowup}
    Let $G$ be a graph with vertex weights $w_i \in [0, \infty)$ and edge weights $p_{ij} \in [0, 1]$. Suppose that $\sum_{i\in V(G)} w_i = 1$ and 
    $$\sum_{\{i,j\} \in E(G)} p_{ij}w_i w_j > \frac{1}{4}.$$
    Let $B(i, j)$ denote the set of common neighbors of vertices $i$ and $j$ in $G$. Then there exists an edge $\{i,j\} \in E(G)$ with $p_{ij} > 0$ such that
    $$\sum_{k \in B(i, j)} p_{ik} p_{jk} w_k \geq \frac{1}{6}.$$
\end{lemma}
We remark that taking $p_{ij} \equiv 1$ and $w_i \equiv \frac{1}{n}$ recovers \cref{lem-Edw}.

\begin{proof}
    Let $N$ be a sufficiently large positive integer. We construct a random graph $\Tilde{G}$ as follows. The vertex set of $\Tilde{G}$ is the disjoint union of independent sets $V_i$ of size $\floor{w_i N}$ for each $i \in V(G)$. The edge set of 
    $\Tilde{G}$ is defined as follows: for each edge $\{i,j\} \in E(G)$, between each pair of vertices in $V_i \times V_j$, we put an edge independently with probability $p_{ij}$. Note that $\Tilde{G}$ has at most $N$ vertices. By Lemma \ref{lem:Cher}, with probability $1 - e^{-\Omega(N)}$, the number of edges in $\Tilde{G}$ is at least 
    $$\sum_{\{i,j\} \in E(G)} p_{ij} \abs{V_i} \abs{V_j} - N^{3/2} > {N^2}/{4}.$$
    Then by Lemma \ref{lem-Edw}, with probability $1 - e^{-\Omega(N)}$, we have $bk(\Tilde{G}) > {N}/{6}$.
    
    By construction, edges in $\Tilde{G}$ can exist between $V_i$ and $V_j$ only if $\{i,j\} \in E(G)$ and $ p_{ij} > 0$. 
    For any $\{i,j\} \in E(G)$ and any two vertices $(u, v) \in V_i \times V_j$, another vertex $w$ is a common neighbor of $u, v$ only if $w \in V_k$ for some $k \in B(i, j)$, in which case $w$ is a common neighbor of $u, v$ with probability $p_{ik} p_{jk}$. Furthermore, for distinct $w$, the events that $w$ is a common neighbor of $u, v$ are independent. Thus, the expected number of common neighbors of $u, v$ is at most $N \cdot \sum_{k \in B(i, j)} p_{ik} p_{jk} w_k$. Then by Lemma \ref{lem:Cher} again, with probability $1 - e^{-\Omega(\sqrt{N})}$, the number of common neighbors of $u, v$ is at most 
    $$N \cdot \sum_{k \in B(i, j)} p_{ik} p_{jk} w_k + N^{3/4}.$$
    Taking the union bound over all pairs of vertices, we conclude that with probability $1 - e^{-\Omega(\sqrt{N})}$, 
    $$N \cdot \max_{\substack{ \{i,j\} \in E(G)\\ p_{ij} > 0}} \sum_{k \in B(i, j)} p_{ik} p_{jk} w_k + N^{3/4} \ge bk(\Tilde{G})>  \frac{N}{6}.$$
    As this holds for any sufficiently large $N$, we have
    $$\max_{\substack{\{i,j\} \in E(G)\\ p_{ij}> 0}} \sum_{k \in B(i, j)} p_{ik} p_{jk} w_k \geq \frac{1}{6}, $$
    as desired.
\end{proof}

The above proof relies on the use of Lemma \ref{lem-Edw}. 
It might be of interest to find an elementary proof of 
Lemma \ref{lem:Edwards-Blowup} without going through Lemma \ref{lem-Edw}. 

We are now ready to prove Theorem \ref{thm-booksize}. 

\begin{proof}[Proof of Theorem \ref{thm-booksize}]
Assume that $G=(V,E)$ is Nosal, i.e. satisfies $\lambda(G) > \sqrt{m}$. Let $\bm{x}$ be the unit Perron--Frobenius eigenvector of $G$. If $G$ is not connected, we can identify one vertex from each connected component to obtain a connected graph with larger spectral radius, the same number of edges and the same booksize. So we may assume that $G$ is connected. This assumption ensures that $x_i > 0$ for every $i\in V$.  

We apply Lemma \ref{lem:Edwards-Blowup} with vertex weights $w_i := x_i^2$, which satisfies the condition $\sum_{i\in V} w_i = 1$.
The choice of the edge weights $p_{ij}$ is more subtle. A straightforward idea is $p_{ij} = 1$ for all $\{i, j\}\in E(G)$. Then by convexity and that $\lambda (G) = 2\sum_{\{i,j\}\in E(G)} x_i x_j$, we have
$$\sum_{\{i,j\} \in E} p_{ij} w_i w_j 
= \sum_{\{i,j\} \in E} x_i^2 x_j^2 \geq \frac{1}{m} \left(\sum_{\{i,j\} \in E} x_i x_j \right)^2 = \frac{\lambda^2(G)}{4m} > \frac{1}{4}.$$
So we can apply Lemma \ref{lem:Edwards-Blowup} to conclude that there exists $\{i,j\} \in E$ such that
$$\sum_{k \in B(i, j)} x_k^2 \geq \frac{1}{6}.$$
Unfortunately, this is not sufficient to lower bound $bk(G)$, since we have no control over $x_k^2$ for $k \in B(i, j)$. Instead, we engineer a different set of edge weights
$$p_{ij} := \frac{\max\left\{m^{-1/2} x_i x_j - (4m)^{-1}, 0\right\}}{x_i^2 x_j^2}.$$
This weight satisfies $p_{ij} \in [0, 1]$ since  $x_i^2 x_j^2 - m^{-1/2} x_i x_j + (4m)^{-1} 
= (x_i x_j - 2^{-1}m^{-1/2})^2 \geq 0$. Furthermore, we have
$$\sum_{\{i,j\} \in E} w_i w_j p_{ij} 
\geq \sum_{\{i,j\} \in E} (m^{-1/2} x_i x_j - (4m)^{-1}) = \frac{\lambda(G)}{2m^{1/2}} - \frac{1}{4} > \frac{1}{4}.$$
So Lemma \ref{lem:Edwards-Blowup} is applicable, 
and there exists $\{i,j\} \in E(G)$ with $p_{ij} > 0$ such that
\begin{equation}  \label{eq-16}
\sum_{k \in B(i, j)} p_{ik} p_{jk} w_k 
\geq \frac{1}{6}.
\end{equation}
Since $p_{ij} > 0$, we have $x_i x_j > 4^{-1}m^{-1/2}$. Without loss of generality, we may assume that 
$$x_i > \frac{1}{2m^{1/4}}.$$
Let $B'(i, j)$ be the set of vertices $k \in B(i, j)$ with $p_{ik} > 0$. Then (\ref{eq-16}) yields 
$$\frac{1}{6} \leq \sum_{k \in B(i, j)} p_{ik} p_{jk} w_k \leq \sum_{k \in B'(i, j)} p_{ik} w_k = \sum_{k \in B'(i, j)} \frac{m^{-1/2} x_i x_k - (4m)^{-1}}{x_i^2}.$$
The subsequent maneuvers are somewhat lossy and do not give an exact bound. First,
$$\frac{1}{6} \leq \sum_{k \in B'(i, j)} \frac{m^{-1/2} x_i x_k - (4m)^{-1}}{x_i^2} < \sum_{k \in B'(i, j)} m^{-1/2} \frac{x_k}{x_i}.$$
Therefore, we have
$$\sum_{k \in B'(i, j)} x_k > \frac{1}{6}x_i m^{1/2} > \frac{1}{12} m^{1/4}.$$
On the other hand, by the Cauchy--Schwarz inequality, we obtain
$$\sum_{k \in B'(i, j)} x_k \leq \sqrt{\abs{B'(i, j)} \sum_{k \in B'(i, j)} x_k^2} \leq \sqrt{\abs{B'(i, j)}}.$$
Thus we conclude that $\abs{B'(i, j)} > \frac{1}{144} \sqrt{m}$, so $B'(i, j)$ and $i, j$ form a book of size at least $\frac{1}{144}\sqrt{m}$, as desired.
\end{proof}

\section{Proof of Theorem \ref{cor:joints}}

\label{sec-3}


An \emph{$r$-joint} is a family of $r$-cliques sharing a common edge. Let $js_r(G)$ be the maximum number of $r$-cliques in an $r$-joint of $G$. For example, we have $js_3(G) = bk(G)$. 
A result of Erd\H{o}s \cite{Erd1969} 
shows that an $n$-vertex graph $G$ with more than 
$e(T_{n,r})$ edges contains not only a $K_{r+1}$, 
but also a large $(r+1)$-joint, i.e., 
$\Omega_r(n^{r-1})$ copies of $K_{r+1}$ sharing an edge. 

\begin{lemma}[See \cite{Erd1969,BN2008}] 
\label{thm-Erd-joint}
    If $G$ is a graph on $n$ vertices with $e(G) > (1 - \frac{1}{r}) \frac{n^2}{2}$, then 
    \[ js_{r+1}(G) = \Omega_r(n^{r - 1}). \] 
\end{lemma}
More precisely, Erd\H{o}s \cite{Erd1969} proved that 
$js_{r+1}(G) \ge n^{r-1}/(10r)^{6r} $.
Bollob\'{a}s and Nikiforov \cite{BN2008} later strengthened the result to $js_{r+1}(G) \ge n^{r-1}/ r^{r+5}$, though it remains an open question to tighten the constant term in Lemma \ref{thm-Erd-joint}. 
In this section, we shall prove the following spectral Tur\'{a}n analogue of Lemma \ref{thm-Erd-joint} for graphs with given size. 

\begin{theorem} \label{thm-jsr+1}
    \label{thm:joints}
    If $G$ is a graph with $\lambda^2(G) > (1 - \frac{1}{r}) 2m$, then 
    \[ js_{r + 1}(G) =
    \Omega_r(m^{\frac{r - 1}{2}}). \] 
    Moreover, this bound is tight up to a constant factor. 
\end{theorem}

In particular, Theorem \ref{thm-booksize} is the special case $r = 2$.  
We shall postpone the proof of Theorem \ref{thm:joints}. 
As a byproduct, we first show that 
Theorem \ref{cor:joints} follows from Theorem \ref{thm:joints}  via an application of the Kruskal--Katona theorem. Moreover, we provide other  applications of 
Theorem \ref{thm:joints}. 

\begin{proof}[Proof of Theorem \ref{cor:joints}]
    Using Theorem \ref{thm:joints}, there exists an edge $\{u,v\}\in E(G)$ such that there are 
     $\Omega_r(m^{\frac{r-1}{2}})$ copies of $K_{r+1}$ in $G$ containing $\{u,v\}$. 
    The Kruskal--Katona theorem 
    implies that every 
    graph with $m$ edges has  $O_t(m^{\frac{t}{2}})$ copies of $K_t$ for every $t\ge 3$; see, e.g., \cite[p. 304]{Bol1978} 
    or \cite{Alon1981}. 
    In particular, 
    there are $O_r(m^{\frac{r-2}{2}})$ copies of $K_{r-2}$ in 
    $G[N(u)\cap N(v)]$. Observe that such a copy of 
    $K_{r-2}$ corresponds to a copy of $K_r$ containing 
    the edge $\{u,v\}$. 
    Thus, 
    there are $O_r(m^{\frac{r-2}{2}})$ copies of $K_r$ in $G$ containing $\{u,v\}$.  
    Recall that there are $\Omega_r(m^{\frac{r-1}{2}})$ copies of 
    $K_{r+1}$ containing $\{u,v\}$. As 
    each copy of $K_{r+1}$ contains $r+1$ copies of $K_r$, we can apply a double counting argument on pairs $(P,Q)$, where $P$ and $Q$ are $r$- and $(r + 1)$-cliques containing $\{u,v\}$ respectively with $P\subseteq Q$. Thus there exist $\Omega_r(\sqrt{m})$ copies of $K_{r+1}$ that share a common $r$-clique. The optimality on the order $\sqrt{m}$ follows from Example \ref{exam-booook} but with $s=\sqrt{(1-\frac{1}{r})2m} +1$ and $t$ about $\frac{m}{r}$ instead.
\end{proof}


In what follows, we provide two more applications of Theorem \ref{thm:joints}. 
As the first one, we obtain the following corollary under the vertex-spectral condition. This corollary recovers a result of Nikiforov \cite{Niki2009ejc} by a quite different method.

\begin{cor}\label{cor:jsr}
    If $G$ is an $n$-vertex graph with 
    $\lambda (G)> (1- \frac{1}{r})n$, 
    then 
    \[ js_{r+1}(G)=\Omega_r(n^{r-1}). \]  
\end{cor}

\begin{proof}
Invoking the fact $\lambda (G)\ge \frac{2m}{n}$, 
we have $\lambda^2(G)
> \frac{2m}{n} \cdot (1- \frac{1}{r})n = (1-\frac{1}{r})2m$. 
Then Theorem \ref{thm:joints} implies that 
$js_{r+1}(G)=\Omega_r(m^{\frac{r-1}{2}})$. 
Since $(1-\frac{1}{r})n < \lambda (G)< \sqrt{2m}$,  
we get $m\ge (1-\frac{1}{r})^2\frac{n^2}{2}$. 
Thus, it follows that $js_{r+1}(G)=\Omega_r(n^{r-1})$. 
\end{proof}

A result of Ning and Zhai \cite{NZ2021} implies that 
if $\lambda (G)> \sqrt{m}$, then $G$ contains 
$\Omega (\sqrt{m})$ triangles. 
As another application of Theorem \ref{thm:joints}, 
we now extend this result to the case of cliques. 

\begin{cor} \label{cor-extend-triangle}
If $G$ is an $m$-edge graph with $\lambda^2 (G)> (1-\frac{1}{r})2m$, then 
$G$ contains $\Omega_r (m^{\frac{r-1}{2}})$ 
copies of $K_{r+1}$, and this bound is tight up to a constant factor. 
\end{cor}

\begin{proof}
    From Theorem \ref{thm:joints}, it follows that $G$ has $\Omega_r (m^{\frac{r-1}{2}})$ copies of $K_{r+1}$. For the tightness, let $G$ be the graph by embedding an edge into a partite set of the complete $r$-partite graph $K_{t,t,\ldots,t}$. 
    Observe that $m={r \choose 2}t^2+1=(1-\frac{1}{r})\frac{n^2}{2}+1$. 
    We claim that $\lambda^2 (G) >(1- \frac{1}{r})2m $. Otherwise, if $\lambda^2(G) \le {(1-\frac{1}{r})2m}$, then combining with $\lambda (G)\ge \frac{2m}{n}$, we obtain $m\le (1-\frac{1}{r})\frac{n^2}{2}$, a contradiction. 
    Note that $t=\sqrt{(m-1)/{r\choose 2}}$. So $G$ contains at most $t^{r-1}=O_r(m^{\frac{r-1}{2}})$ copies of $K_{r+1}$. 
\end{proof}

\subsection{Proof of Theorem \ref{thm:joints}}

Extending the proof of Theorem \ref{thm-booksize}, 
we replace Hoeffding's inequality with the Kim--Vu polynomial concentration inequality \cite{KV2000}, which is a celebrated result for proving the concentration of subgraph counts. We state a directly applicable corollary of Kim--Vu's inequality.
\begin{lemma}
\label{lem:KV-corollary}
Let $V$ be a set of at most $N$ vertices. Let $u, v$ be any two distinct vertices in $V$, and let $\C C$ be a family of $(r + 1)$-element sets such that each $C \in \C C$ contains $u$ and $v$. Let $G$ be a random graph on $V$ where each edge is included independently, possibly with different probability. Let $X$ be the random variable counting the number of $C \in \C C$ such that $G[C]$ is a clique. Let $\EE'(X)$ denote the expectation of $X$ conditioned on $\{u,v\} \in E$. We have
$$\mathrm{Pr}\left(X - \EE'(X) \geq N^{r - 5/4} \right) \leq \exp(-N^{\Omega_r(1)}).$$
\end{lemma}
\begin{proof}
    Note that $\abs{C} \leq N^{r - 1}$. If $\{u,v\} \notin E$, we have $X = 0$. So we condition on the event $\{u,v\} \in E$. 
   We denote $k = \binom{r + 1}{2} - 1$ and  consider the $k$-uniform hypergraph $H$, with vertex set $V(H) = \binom{V}{2} - \{u, v\}$ and edge set $\C E(H) = \{\{e \in V(H): e \subset C\}: C \in \C C\}$. Conditioned on $\{u,v\} \in E(G)$, we have
    $$X = \sum_{e \in E(H)} \prod_{i \in e} t_i$$
    where $t_i$ is the indicator variable denoting whether edge $i \in \binom{V}{2}$ lies in $G$. By assumption, $t_i$ are independent Bernoulli random variables. Furthermore, for any $i \in V(H)$, the number of $C \in \C C$ containing $i$ is at most $N^{r - 2}$, since any such $C$ must contain vertices $u, v$ and at least one other element of $i$. In the language of \cite{KV2000}, we have
    $$\EE(X) \leq N^{r - 1}, \quad \EE'(X) \leq N^{r - 2}.$$
    Applying the main theorem of \cite{KV2000} with $\lambda = N^{1 / 100k}$, we conclude the desired result.
\end{proof}

The proof of Theorem \ref{thm:joints} 
proceeds via the following weighted version of Lemma \ref{thm-Erd-joint}, whose proof is analogous to Lemma \ref{lem:Edwards-Blowup}. 
For completeness, we shall include a detailed proof. 

\begin{lemma}
    \label{lem:joints-Blowup}
    Let $G=(V,E)$ be a graph with vertex weights $w_i \in [0, \infty)$ and edge weights $p_{ij} \in [0, 1]$. Suppose that $\sum_{i\in V} w_i = 1$ and 
    $$\sum_{\{i,j\} \in E} p_{ij}w_i w_j > \frac{r - 1}{2r}.$$
    Then there exists an edge $\{i,j\} \in E$ with $p_{ij} > 0$, such that if $\C C(i, j)$ denotes the set of $(r + 1)$-cliques in $G$ containing $\{i, j\}$, then
    $$\sum_{C \in \C C(i, j)} \prod_{k \in C\setminus \{i,j\}} w_{k} 
    \prod_{\substack{\{s,t\}\subseteq C \\
      \{s,t\} \neq \{i,j\} }} p_{st} = \Omega_r(1).$$
\end{lemma}
Taking $p_{ij}\equiv 1$ and $w_i \equiv {1}/{n}$ recovers Lemma \ref{thm-Erd-joint}.

\begin{proof} 
Let $N$ be a sufficiently large integer. 
    We construct a random graph $\Tilde{G}$ as follows. Each vertex $i \in V(G)$ is blown up into an independent set $V_i$ of size $\floor{w_i N}$, and the edge set of $\Tilde{G}$ is as follows: for each $\{i,j\} \in E(G)$, we put an edge independently with probability $p_{ij}$ between each pair of vertices in $V_i$ and $V_j$. Note that $\Tilde{G}$ has at most $N$ vertices. By~\Cref{lem:Cher}, with probability $1 - e^{-\Omega(N)}$ the number of edges in $\Tilde{G}$ is at least $\sum_{\{i,j\} \in E(G)} p_{ij} \abs{V_i} \abs{V_j} - N^{3/2} 
    > \frac{r-1}{2r}N^2.$
    By~\Cref{thm-Erd-joint}, 
    with probability $1 - e^{-\Omega(N)}$ 
    we have $js_{r+1}(\Tilde{G}) = \Omega_r (N^{r-1})$. 
    
    On the other hand, note that edges in $\Tilde{G}$ can exist between $V_i$ and $V_j$ only if $\{i,j\} \in E(G)$ and $ p(i, j) > 0$. 
    Consider any $\{i,j\} \in E(G)$ and $(u,v)\in V_i\times V_j$. Conditioned on $\{u, v\} \in E(\Tilde{G})$, the expected number of $(r+1)$-cliques in $\Tilde{G}$ containing $\{u, v\}$ is
    $$\sum_{C \in \C C(i, j)} \prod_{k\in C\setminus \{i,j\}} \abs{V_k} \cdot \prod_{ \substack{\{s,t\}\subseteq C \\
      \{s,t\} \neq \{i,j\} } } p_{st}
    \leq N^{r-1} \cdot \sum_{C \in \C C(i, j)}  
    \prod_{k\in C\setminus \{i,j\}} w_k \cdot 
    \prod_{\substack{\{s,t\}\subseteq C \\
      \{s,t\} \neq \{i,j\} } } p_{st}.$$
  We apply~\Cref{lem:KV-corollary}. With probability $1 - \exp(-N^{\Omega_r(1)})$, the number of $(r+1)$-cliques in $\Tilde{G}$ containing $\{u, v\}$ is at most 
    $$N^{r-1} \cdot \sum_{C \in \C C(i, j)}  
    \prod_{k\in C\setminus \{i,j\}} w_k \cdot 
    \prod_{\substack{\{s,t\}\subseteq C \\
      \{s,t\} \neq \{i,j\} } } p_{st} + O(N^{r- \frac{5}{4}}).$$
    Taking the union bound over all pairs of vertices, with probability $1 - \exp(-N^{\Omega_r(1)})$ we have
    $$N^{r-1} \cdot 
    \max_{\substack{\{i,j\}\in E(G)\\ p_{ij}>0}} 
    \sum_{C \in \C C(i, j)}  
    \prod_{k\in C\setminus \{i,j\}} w_k \cdot 
    \prod_{\substack{\{s,t\}\subseteq C \\
      \{s,t\} \neq \{i,j\} } } p_{st} + O(N^{r-5/4})\ge js_{r+1}(\Tilde{G})
    > \Omega_r(N^{r-1}).$$
    Since this holds for any sufficiently large $N$, we 
    obtain the desired bound. 
\end{proof}

Now, we present the proof of Theorem \ref{thm-jsr+1}.

\begin{proof}[Proof of Theorem \ref{thm-jsr+1}] 
Analogous to the proof of Theorem \ref{thm-booksize} we may assume that $G$ is connected, so $x_i > 0$ for every $i \in V$. 

We define the vertex weights $w_i := x_i^2$ and the edge weights 
$$p_{ij} := \frac{\max\left\{ \sqrt{2(r - 1)/(rm)} x_i x_j - (r - 1)/(2 r m), 0\right\}}{x_i^2 x_j^2}.$$
By the assumption $\lambda (G) > \sqrt{2m(r-1)/r } $, it is easy to check that
$$\sum_{\{i, j\} \in E} p_{ij} w_i w_j 
\ge \sqrt{\frac{2(r-1)}{rm}} \sum_{\{i,j\}\in E} x_ix_j 
- \frac{r-1}{2r} 
> \frac{r - 1}{2r}.$$
By Lemma \ref{lem:joints-Blowup}, there exist two vertices $i, j$ with $\{i,j\} \in E(G)$ and $p_{ij} > 0$ such that
$$\sum_{C \in \C C(i, j)} \prod_{k \in C\setminus \{i,j\}} w_{k} 
\prod_{\substack{\{s,t\}\subseteq C \\
      \{s,t\} \neq \{i,j\} }} p_{st} =  \Omega_r(1).$$
As $p_{ij} > 0$,  we have 
$x_ix_j = \Omega_r(m^{-1/2})$. Thus, 
 we may assume that 
\[ x_i = \Omega_r( m^{-1/4}). \] 
 For each $C \in \C C(i, j)$, we have 
$$\prod_{k \in C\setminus \{i,j\}} w_{k} 
\prod_{\substack{\{s,t\}\subseteq C \\
      \{s,t\} \neq \{i,j\} }} p_{st} \leq 
      \prod_{k \in C\setminus \{i,j\}} w_{k}p_{ik} = O_r(1) 
      \prod_{k \in C\setminus \{i,j\}} \frac{x_k}{\sqrt{m}x_i}.$$
Then we have
$$\sum_{C \in \C C(i, j)} \prod_{k \in C\setminus \{i,j\}} 
x_k \ge \Omega_r(1) \cdot (\sqrt{m} x_i)^{r-1} 
= \Omega_r(m^{\frac{r-1}{4}}).$$
By the Cauchy--Schwarz inequality, we have
$$\left(\sum_{C \in \C C(i, j)} \prod_{k \in C\setminus \{i,j\}} x_k\right)^2 \leq 
\abs{C(i, j)} \left(\sum_{C \in \C C(i, j)} 
\prod_{k \in C\setminus \{i,j\}} x_k^2\right)\le \abs{C(i, j)}\left(\sum_{k} x_k^2\right)^{r-1} = \abs{C(i, j)},$$
where we used $\sum_{k} x_k^2 =1$. 
So we conclude that
$\abs{C(i, j)} \ge \Omega_r (m^{\frac{r - 1}{2}})$, 
as desired. 

The tightness of $js_{r+1}(G)=\Omega_r(m^{\frac{r-1}{2}})$ can be seen by modifying the graphs in Example \ref{exam-booook}. We
consider the graph $G=K_{s}\circ H$ with $s=\sqrt{(1-\frac{1}{r})2m} +1$ and $t\approx \frac{m}{r}$. Clearly, we have 
$\lambda (G)> \sqrt{(1- \frac{1}{r})2m}$, and $G$ has $O_r(m^{\frac{r-1}{2}})$ copies of $K_{r+1}$ sharing a common edge. 
\end{proof}

\section{Proofs of Theorems \ref{thm:lim} -- \ref{lem:Nosal-structure}}


\label{sec-4} 

Throughout this section, let $\# C_4$ denote the number of copies of $C_4$ in the Nosal graph $G$. For vertex sets $A$, let $G[A]$ denote the induced subgraph, and let $E[A]$ be the edge set of $G[A]$. Similarly, for disjoint vertex sets $A$ and $B$, let $G[A, B]$ denote the induced bipartite subgraph, and let $E[A, B]$ be the edge set of $G[A, B]$.

To prove \Cref{thm:lim}, 
we first present two constructions that show the upper bound 
\[ f(m) \leq \left( \frac{1}{8} + o(1)\right)m^2. \] 

\begin{example} \label{exam-C4-1}
Let $s= \ceil{\sqrt{m}} +1$  and $t=m- {s \choose 2}$.  
Let $G = K_s\circ K_{1,t}$ be the graph obtained 
from the complete graph $K_s$ by adding $t$ pendant edges to one of the vertice. We have $\lambda(G) > \lambda (K_s) = s - 1 \geq \sqrt{m}$. 
 Observe that $G$ contains ${s \choose 4}$ 
 copies of $K_4$ and each $K_4$ contains three copies of $C_4$. 
 Thus the number of copies of $C_4$ in $G$ is $3 {s\choose 4} =  \frac{1}{8}m^2 + O(m^{3/2})$, where the error term oscillates between $\pm \frac{1}{4} m^{3/2}$. In fact, we can start with the complete graph $K_{s}$ and distribute the remaining $t \approx \frac{m}{2}$ edges arbitrarily 
 as long as they do not produce a new copy of $C_4$.  
 \end{example}

\begin{example}  \label{exam-C4-2}
Assume $m$ is odd. Let $K_2 \vee \frac{m-1}{2}K_{1}$ be the graph obtained from 
an edge $K_2$ and an independent set
$\frac{m-1}{2}K_1$ 
by joining each vertex of $K_2$ to each vertex of $\frac{m-1}{2}K_1$. Upon computation, we know that 
$\lambda (K_2 \vee \frac{m-1}{2}K_1) = \frac{1+\sqrt{4m-3}}{2} > 
\sqrt{m}$. For this graph, we have $\# C_4={\frac{m-1}{2} \choose 2}= \frac{1}{8}(m-1)(m-3)$.  

In fact,  for every $m$ there is a similar construction with fewer copies of $C_4$.
Let $G$ be the graph obtained from 
$K_2\vee \frac{m-t-1}{2}K_1$ by adding $t$ pendent edges to 
one vertex of $K_2$. 
Then $\lambda (G) > 
\lambda (K_2\vee \frac{m-t-1}{2}K_1) = 
\frac{1}{2}\bigl(1+\sqrt{4(m-t)-3}\, \bigr) > \sqrt{m}$ for every $m> (t+1)^2$. Taking $t$ to be the largest integer less than $\floor{\sqrt{m}-1}$ such that $m - t - 1$ is even, we have $\#C_4 = {\frac{m-t-1}{2} \choose 2}= 
\frac{1}{8}m^2 - \frac{1}{4}m^{3/2} + O(m)$. We conclude the upper bound
$f(m) \leq \frac{1}{8}m^2 - \frac{1}{4}m^{3/2} + O(m)$. 
\end{example}

We now prove the lower bound part of Theorem \ref{thm:lim} assuming~\Cref{lem:Nosal-max-degree,lem:Nosal-structure}. As outlined before, we shall use the fourth moment of the eigenvalues of $G$ to count $C_4$, as shown in the following lemma.
\begin{lemma}[See \cite{LL2009}] 
\label{lem:LL-counting}
Let $G$ be a graph on $n$ vertices. Writing $M(G):= \sum_{i \in V} d_i^2$, we have 
\[ \# C_4 = \frac{1}{8}\sum_{i\in V}  (\lambda_i^4 +\lambda_i^2)
       - \frac{1}{4} M(G). \]
\end{lemma}

Here we include a proof for the convenience of readers. 

\begin{proof}
Note that the matrix entry $(A^4)_{i,j}$ is equal to the number of walks of length $4$ from vertex $i$ 
to $j$. Thus, we obtain $\mathrm{Tr}(A^4)=8\# C_4 + 
2\# K_2 + 4 \#K_{1,2}$, where $\# H$ denotes the number of copies of $H$ in $G$. Moreover, we have 
$\sum_{i\in V} d_i^2 = 2 \# K_2 + 2 \#K_{1,2}$. 
Therefore, we have 
\[ \# C_4= \frac{1}{8} \left(  \mathrm{Tr}(A^4) +2m - 
2 \sum_{i\in V} d_i^2 \right) 
= \frac{1}{8} \sum_{i\in V} (\lambda_i^4 + \lambda_i^2 - 2 d_i^2). \] 
The desired result follows immediately. 
\end{proof}
However, if we employ Lemma \ref{lem:LL-counting} with the trivial estimates, we obtain
\begin{equation}\label{eq:trace-not-enough}
  \#C_4 > \frac{1}{8}\left(\lambda^4(G) + \lambda^2(G) \right)
       - \frac{1}{4} M(G) > \frac{1}{8}(m^2 + m)
       - \frac{1}{4} (m^2 + m),   
\end{equation}
which is not even positive. We shall use~\Cref{lem:Nosal-max-degree,lem:Nosal-structure} to improve this estimate.

We first translate Theorem \ref{lem:Nosal-max-degree} into an estimate on $M(G)$.

\begin{lemma}
\label{lem:MG}
Let $M(G) = \sum_{i \in V} d_i^2$. Then $ M(G) \leq \Delta(G) \cdot m + o(m^2).$
\end{lemma}

\begin{proof}
    Let $B$ be the set of vertices in $G$ with degree at least $m^{0.7}$. Then $\abs{B} \leq 2m^{0.3}$ and 
    $$\sum_{i \in B} d_i \le e(G)+e(G[B])\le m + \binom{\abs{B}}{2} \leq m + 2m^{0.6}.$$
    Note that $\Delta (G)\le m$. Thus we have
    $$\sum_{i \in B} d_i^2 \leq \Delta(G) 
    \cdot (m + 2m^{0.6} ) \leq \Delta(G) \cdot m + 2m^{1.6}.$$
    On the other hand, by the definition of $B$, we have  
    $$\sum_{i \notin B} d_i^2 \leq m^{0.7} \sum_{i \notin B} d_i = 2m^{1.7}.$$
    Consequently, we obtain 
    $M(G)\le \Delta (G)\cdot m + 4m^{1.7}$. 
\end{proof}

Combining with Theorem \ref{lem:Nosal-max-degree}, we get the promised estimate on $M(G)$. 

\begin{cor}
    \label{cor:MG}
    If $G$ is Nosal with $m$ edges, then we have
    \[  M(G)= \sum_{i \in V} d_i^2 \leq \left( \frac{1}{2} + o(1) \right) m^2. \] 
\end{cor}

We are now ready to establish Theorem \ref{thm:lim}.

\begin{proof}[Proof of Theorem \ref{thm:lim}] 
For any $\varepsilon > 0$, let $G'$ be the subgraph of $G$ given by Theorem \ref{lem:Nosal-structure}, and let $\#C_4'$ denote the number of copies of $C_4$ in $G'$. 

If $G'$ is bipartite, then its eigenvalues 
are symmetric with respect to the origin. Therefore, the smallest eigenvalue of $G'$ is at most $-(1 - \varepsilon - o(1)) \sqrt{m}$. By Corollary \ref{cor:MG}, we get 
\[  M(G') \le M(G) \le \left(\frac{1}{2} + o(1) \right)m^2. \]
Applying Lemma \ref{lem:LL-counting} to $G'$, we conclude that
$$\#C_4' \geq \frac{1}{8}\cdot \left(2(1 - \varepsilon)^4 - o(1) \right) m^2 - \frac{1}{4} M(G') \geq \frac{1}{8}\cdot \big(2(1 - \varepsilon)^4 - 1 - o(1) \big) m^2 
= \left( \frac{1}{8} - O(\varepsilon) \right) m^2.$$

If $G'$ has maximum degree at most $\varepsilon m$, then by Lemma \ref{lem:MG}, we have
\[ M(G') \le \varepsilon m^2 + o(m^2). \]
So we conclude that
$$\#C_4' \geq \frac{1}{8}\cdot \left( (1 - \varepsilon)^4 - o(1) \right) m^2 - \frac{1}{4} M(G') = \left( \frac{1}{8} - O(\varepsilon) \right) m^2.$$
In either case, the number of $4$-cycles in $G$ is at least $(1 /8 - O(\varepsilon)) m^2$. As $\varepsilon > 0$ is arbitrary, we conclude that $\#C_4 \geq (1 /8 - o(1)) m^2$.
\end{proof}


\subsection{The maximum degree of Nosal graphs}

\begin{proof}[Proof of~Theorem \ref{lem:Nosal-max-degree}]
Let $i_*$ be the vertex that maximizes $x_{i_*}$. We repeat the following operation: if there is an edge $jk$ such that $j$ is not adjacent to $i_*$, then remove this edge from $G$ and add the edge $i_* j$. Note that this operation does not decrease $\bm{x}^T A_G \bm{x}$, so it preserves the Nosal property. When we can no longer do this operation, remove all isolated vertices, and let $G_*$ be the resulting graph. Then $G_*$ is Nosal, and $i_*$ is universal in $G_*$, that is, $i_*$ is adjacent to all the other vertices. Observe also that $\Delta(G_*) \geq \Delta(G)$. 

Thus, we may assume that our Nosal graph $G$ has a universal vertex $v$. Set $k = d_v$. For the sake of contradiction, 
we assume that $k > m/2 + m^{0.99}$.

Since $\lambda(G) > \sqrt{m}$, we have
$$ 2\sum_{\{i,j\} \in E} x_i x_j > \sqrt{m} 
\sum_{i \in V} x_i^2.$$
Let $G' = (V', E')$ be the subgraph of $G$ induced by 
$V(G) \backslash \{v\}$. 
Then $|V'|=k > m/2 +m^{0.99}$ and 
$|E'|=m-k$. 
We write $d_i'$ and $N'(i)$ for the degree and the neighborhood of a vertex $i\in V'$, respectively. 
Isolating the terms with $i = v$ in the above inequality, we have
$$2 \sum_{\{i,j\} \in E'} x_i x_j + 
2 \sum_{i \in V'} x_i x_v
> \sqrt{m} \sum_{i \in V'} x_i^2 + \sqrt{m} x_v^2.$$
Using the AM-GM inequality, we have
$$\sqrt{m} x_v^2 + \frac{1}{\sqrt{m}} 
\left( \sum_{i\in V'} x_i \right)^2 
\ge 2\sum_{i\in V'} x_i x_v.$$
Summing these two inequalities, we get 
\begin{equation} \label{eq-ge}
 2\sum_{\{i,j\} \in E'} x_i x_j + 
\frac{1}{\sqrt{m}} \left(\sum_{i \in V'} x_i \right)^2 
> \sqrt{m} \sum_{i \in V'} x_i^2. 
\end{equation} 
We now employ a classical estimate on 
$\sum_{\{i,j\} \in E'} x_i x_j$. Note that
$$ 2 \sum_{\{i,j\} \in E'} x_i x_j 
\leq \sum_{\{i,j\} \in E'} 
\left( \frac{\sqrt{d_j' + 1}}{\sqrt{d_i' + 1}} x_i^2 + \frac{\sqrt{d_i' + 1}}{\sqrt{d_j' + 1}} x_j^2\right).$$
For every $i$, we set
$$q_i := \sum_{j\in N'(i)} \frac{\sqrt{d_j' + 1}}{\sqrt{d_i' + 1}}.$$  
Then (\ref{eq-ge}) gives 
\begin{equation} \label{eq-recall}
\frac{1}{\sqrt{m}} \left(\sum_{i \in V'} x_i\right)^2 > 
\sum_{i \in V'} (\sqrt{m} - q_i ) \cdot x_i^2.
\end{equation}
We can bound $q_i$ as follows.

\begin{claim}\label{cl:1}
$q_i < \min\left\{ 
\sqrt{2(m - k)}, \sqrt{ \frac{d_i'}{d_i' + 1} \left((m - k) + d_i'^2\right)} \, \right\}.$   
\end{claim}

\begin{poc}
Using the Cauchy--Schwarz inequality, we obtain 
$$q_i \leq \sqrt{\frac{d_i'}{d_i' + 1}} \cdot \sqrt{\sum_{j \in N'(i)} (d_j' + 1)}.$$
Let $S$ denote the vertex set consisting of $i$ and its neighbors in $V'$. 
Then $\sum_{j \in N'(i)} (d_j' + 1)=\sum_{j\in S}d_j'$ is equal to twice the number of edges in $G'[S]$ plus the number of edges in $G'$ with exactly one vertex in $S$. In particular, this implies
$$\sum_{j \in N'(i)} (d_j' + 1) \leq 
\min\left\{2e(G'), e(G') + 
\binom{\abs{S}}{2}\right\}.$$
The claim follows as $e(G') = m - k$ and $\abs{S} = d_i' + 1$.  
\end{poc}

Recall that we assume $k > m/2 + m^{0.99}$. 
By Claim \ref{cl:1}, for each $i \in V'$ we have
$$\sqrt{m} - q_i > 0. $$ 
Applying the Cauchy--Schwarz inequality on the right hand side of~\eqref{eq-recall} gives 
$$\frac{1}{\sqrt{m}} \left(\sum_{i \in V'} x_i\right)^2>\sum_{i \in V'} (\sqrt{m} - q_i) \cdot x_i^2 \geq \left(\sum_{i \in V'} x_i \right)^2 \cdot \frac{1}{\sum\limits_{i \in V'} \frac{1}{\sqrt{m} - q_i}}.$$
So we must have
\begin{equation} \label{eqge-sqm}
    \sum_{i \in V'} \frac{1}{\sqrt{m} - q_i} > \sqrt{m}.
\end{equation}
Let $S_1$ be the set of vertices in $V'$ with degree at least $m^{0.25}$, and let $S_2 = V' \backslash S_1$.  
For any vertex $i \in S_1$, Claim \ref{cl:1} implies that  
\begin{equation}  \label{eq-S1} 
\frac{1}{\sqrt{m} - q_i} < \frac{1}{\sqrt{m} - \sqrt{2(m - k)}} < \frac{1}{\sqrt{m} - \sqrt{m - 2m^{0.99}}} < m^{-0.49}. 
\end{equation} 
For each vertex $i \in S_2$, by definition 
we have $d_i' < m^{0.25}$. 
Next, we prove a more sophisticated estimate 
on $1/(\sqrt{m} - q_i)$ for every $i\in S_2$. Let $d'_{i,1}$ denote the number of edges between $i$ and vertices in $S_1$, and $d'_{i,2}$ denote the number of edges between $i$ and vertices in $S_2$. 

\begin{claim}\label{cl:2}
    For any vertex $i\in S_2$, we have 
\begin{equation} \label{eq-S2}
\frac{1}{\sqrt{m} - q_i} \leq m^{-0.5} + (m^{-0.5} - m^{-0.6}) d'_{i,1} + m^{-0.6} d'_{i,2}.
\end{equation} 
\end{claim}

\begin{poc}
Using Cauchy-Schwarz's inequality 
as in Claim \ref{cl:1}, we have
$$\sum_{j \in S_1 \cap N'(i)} \frac{\sqrt{d_j' + 1}}{\sqrt{d_i' + 1}} \leq \sqrt{ \frac{d'_{i,1}}{d_i' + 1} \left((m - k) + d_i'^2\right)} \leq 
\sqrt{ \frac{d'_{i,1}}{d'_{i,1} + 1} 
\left( \frac{m}{2} - m^{0.98} \right)},$$
where the last inequality follows from 
$k> m/2 + m^{0.99}$ and $d_i'^2 < m^{0.5}$. 
Furthermore, we have
$$\sum_{j \in S_2\cap N'(i)} \frac{\sqrt{d_j' + 1}}{\sqrt{d_i' + 1}} \leq \sqrt{m^{0.25} + 1} \cdot \frac{d_{i,2}'}{\sqrt{d_i' + 1}} \leq m^{0.2} d_{i,2}'.$$
So for every $i\in S_2$, it follows that  
\[ q_i \leq \sqrt{ \frac{d'_{i,1}}{d'_{i,1} + 1} 
\left( \frac{m}{2} - m^{0.98} \right)} + m^{0.2} d_{i,2}' . \]
Then we have
$$\frac{1}{\sqrt{m} - q_i} \leq \frac{1}{\sqrt{m} - \sqrt{\frac{d'_{i,1}}{d'_{i,1} + 1} 
\left( \frac{m}{2} - m^{0.98} \right)} - m^{0.2} d_{i,2}'}.$$ 
Note that the difference between the first two terms in the denominator above is greater than $(1 - \sqrt{1 / 2})\sqrt{m}$, 
which together with $d_{i,2}' < d_i' < m^{0.25}$ 
yields 
$$\frac{1}{\sqrt{m} - q_i} \leq 
\frac{1}{\sqrt{m} - \sqrt{\frac{d'_{i,1}}{d'_{i,1} + 1} 
\left( \frac{m}{2} - m^{0.98}\right)}} + 100m^{-0.8} d_{i,2}'.$$
One can now directly verify
$$\frac{1}{\sqrt{m} - \sqrt{\frac{d'_{i,1}}{d'_{i,1} + 1} 
\left( \frac{m}{2} - m^{0.98}\right)}} \leq 
m^{-0.5} + (m^{-0.5} - m^{-0.6}) d'_{i,1}.$$
Indeed, if $d'_{i,1} \in \{0, 1\}$, then this follows by direct computation. If $d'_{i,1} \geq 2$ we have
$$\frac{1}{\sqrt{m} - \sqrt{\frac{d'_{i,1}}{d'_{i,1} + 1} \left( \frac{m}{2} - m^{0.98} \right)}} 
<  \frac{1}{\sqrt{m}} \cdot \frac{1}{1 - \sqrt{\frac{d'_{i,1}}{2(d'_{i,1} + 1)}}} 
< 0.9 m^{-0.5} (1 + d'_{i,1}),$$
where the last inequality follows since for any real number $x> 1.43$, 
\[ 0.9\left( 1-\sqrt{\frac{x}{2(x+1)}}\, \right)(1+x) >1. \]  
So we have verified the inequality (\ref{eq-S2}).  
\end{poc}

Combining (\ref{eq-S1}) with (\ref{eq-S2}), we have
\begin{align*}
 \sum_{i \in V'} \frac{1}{\sqrt{m} - q_i} &\leq m^{-0.49} \abs{S_1} + \sum_{i \in S_2} \left( 
m^{-0.5} + (m^{-0.5} - m^{-0.6}) d'_{i,1} + m^{-0.6} d'_{i,2} \right) \\
&= m^{-0.5} \abs{S_2} + m^{-0.49} \abs{S_1} + \sum_{i \in S_2} \left((m^{-0.5} - m^{-0.6}) d'_{i,1} + m^{-0.6} d'_{i,2} \right)
\end{align*}
We now bound the right hand side using the discharging method. For each edge $e \in E'$ and endpoint $i \in e$, we define a weight $w_{e, i}$ that $e$ ``discharges" to $i$ as follows
\begin{enumerate}
    \item If $e$ lies in $G'[S_1]$ or $G'[S_2]$, we set $w_{e, i} = m^{-0.6}$.
    \item Otherwise, let the endpoints of $e$ be $i, j$ with $i \in S_1$ and $j \in S_2$. We set $w_{e, i} = m^{-0.7}$ and $w_{e, j} = m^{-0.5} - m^{-0.6}$.
\end{enumerate}
Recall that $S_1$ denotes the set of vertices with degree at least $m^{0.25}$.  On the one hand, each vertex $i$ in $S_1$ satisfies
$$\sum_{e: i \in e} w_{e, i} \geq m^{-0.7} \cdot  m^{0.25} > m^{-0.49}$$
while each vertex $i$ in $S_2$ satisfies
$$\sum_{e: i \in e} w_{e, i} \geq (m^{-0.5} - m^{-0.6}) d'_{i,1} + m^{-0.6} d'_{i,2}.$$ 
On the other hand, each edge $e$ in $E'$ satisfies
\[ \sum_{i: i \in e} w_{e, i} \leq \max\{2m^{-0.6}, m^{-0.7} + m^{-0.5} - m^{-0.6}\} < m^{-0.5} - m^{-0.7} .\] 
So we conclude that
$$m^{-0.49} \abs{S_1} + \sum_{i \in S_2} (m^{-0.5} - m^{-0.6}) d'_{i,1} + m^{-0.6} d'_{i,2} 
\leq \sum_{i \in V', e \in E': i \in e} w_{e, i} \leq (m^{-0.5} - m^{-0.7}) \abs{E'}.$$
Invoking $|E'|=m-k$ and $|S_2| 
\le |V'| = k$, 
 we have
$$\sum_{i \in V'} \frac{1}{\sqrt{m} - q_i} \leq m^{-0.5} k + (m^{-0.5} - m^{-0.7})(m - k) \leq \sqrt{m},$$
which contradicts with (\ref{eqge-sqm}). 
Therefore, we must have $k \le m/2 + m^{0.99}$. 
\end{proof}

In the above proof, under the assumption $k> m/2 + m^{0.99}$, 
we have shown that \eqref{eq-ge} does not hold. In other words, we have shown the following result, which may be of independent interest.

\begin{proposition} 
Let $m$ and $k$ be integers such that 
$k>m/2 + m^{0.99}$. 
If $G'=(V',E')$ is a graph with 
$|V'|=k$ and $|E'|=m - k$, then 
$$ 2\sum_{\{i,j\} \in E'} x_i x_j \leq \sqrt{m} \sum_{i \in V'} x_i^2 - \frac{1}{\sqrt{m}} 
\left(\sum_{i \in V'} x_i \right)^2.$$ 
\end{proposition}

\subsection{Structural dichotomy for Nosal graphs}

\begin{proof}[Proof of~Theorem \ref{lem:Nosal-structure}]
    Let $\lambda = \lambda(G)$. Without loss of generality, assume $\varepsilon < 0.1$. 
    
    The crucial idea is to study the following sequence of sets. 
    Recall that $\bm{x}$ is the Perron-Frobenius eigenvector of $G$, which satisfies 
    $\sum_{i\in V} x_i^2 =1$.  By assumption, we have
    $$\sum_{\{i,j\} \in E} 2x_i x_j = \lambda > \sqrt{m}.$$
    For each positive integer $t$, we define the following sets
    $$A_t := \{i\in V: x_i \geq \varepsilon^{2t}\},$$
    $$B_t := \{i\in V: x_i \leq \varepsilon^{-2t} m^{-1/2}\},$$
    $$C_t := \{i\in V: \varepsilon^{-2t} m^{-1/2} < x_i < \varepsilon^{2t}\}.$$
    Observe that $A_t \sqcup B_t \sqcup C_t = V$ and $C_t \subset C_{t-1}$. Note that the bipartite graphs $G[C_t, C_{t - 1} \backslash C_t]$ are edge disjoint. By the pigeonhole principal, there exists some $t < \varepsilon^{-2} + 2$ such that
    $$\sum_{\{i,j\} \in E[C_t, C_{t - 1} \backslash C_t]} 
    2 x_ix_j \leq \varepsilon^2 \lambda.$$
    Since $\sum_i x_i^2 =1$, we have $\abs{A_{t - 1}} \leq \varepsilon^{-4(t - 1)}$ and $|C_t| \le \varepsilon^{4t}m$. 
    Thus we have
    $$\sum_{\{i,j\} \in E[C_t, A_{t - 1}]} 2x_i x_j\le \lambda (G[C_t, A_{t-1}])\le \sqrt{|C_t|\cdot |A_{t-1}|} \le \varepsilon^{2}\sqrt{m} 
    \le \varepsilon^{2}\lambda.$$
    For each $\{i,j\} \in E[C_t, B_{t - 1}]$, we have
    $$x_i x_j \leq \varepsilon^{-2(t - 1)}m^{-1/2} \cdot \varepsilon^{2t} \leq \varepsilon^2 m^{-1/2}.$$
    So we get
    $$\sum_{\{i,j\} \in E[C_t, B_{t - 1}]} 2x_i x_j \leq  \varepsilon^2 m^{-1/2} \cdot 2m \leq 2\varepsilon^2 \lambda.$$ 
    Recall that $V=A_{t-1} \sqcup B_{t-1} \sqcup C_{t-1}$ and $C_t \subseteq C_{t-1}$. 
    Thus we have $\overline{C_t} = A_{t - 1} \cup B_{t - 1} \cup (C_{t - 1} \backslash C_t)$. Combining the inequalities above, we conclude that 
    $$\sum_{\{i,j\} \in E[C_t, \overline{C_t}]} 2x_i x_j\leq 4 \varepsilon^2 \lambda \leq \varepsilon \lambda.$$
    Thus we have
    \begin{equation}
    \label{eq:Nosal-structure-intermediate}
    \sum_{\{i,j\} \in E[C_t]} 2x_i x_j 
    + \sum_{\{i,j\} \in E[\overline{C_t}]}2 x_i x_j \geq (1 - \varepsilon) \lambda.
    \end{equation}
    Note that $\overline{C_t} = A_t \cup B_t$. Since $\abs{A_{t}} \leq \varepsilon^{-4t}$, we obtain 
    $$\sum_{\{i, j\} \in E[A_t]} 2x_i x_j \leq \abs{A_t}^2 \leq \varepsilon^{-8t}.$$
    By the definition of $B_t$, we have
    $$\sum_{\{i,j\} \in E[B_t]} 2x_i x_j \leq 2m \cdot \left(\varepsilon^{-2t} m^{-1/2}\right)^2 \leq 2 \varepsilon^{-4t}.$$
    In conclusion, subtracting from \eqref{eq:Nosal-structure-intermediate} the contribution of $E[A_t]$ and $E[B_t]$, we have
    $$\sum_{\{i,j\} \in E[C_t]} 2x_i x_j 
    + \sum_{\{i,j\} \in E[A_t, B_t]} 2x_i x_j \geq (1 - \varepsilon) \lambda - 3 \varepsilon^{-8t}.$$
    In particular, setting $N(\varepsilon) = 3 \cdot \varepsilon^{-8(\varepsilon^{-2} + 2)}$, the disjoint union of $G[C_t]$ and $G[A_t, B_t]$ has spectral radius at least $(1 - \varepsilon) \lambda - N(\varepsilon)$. So one of the two subgraphs has spectral radius at least $(1 - \varepsilon) \lambda - N(\varepsilon)$.

    If $G[A_t, B_t]$ has spectral radius at least $(1 - \varepsilon) \lambda - N(\varepsilon)$, then it satisfies (a), as desired. Otherwise, $G[C_t]$ has spectral radius at least $(1 - \varepsilon) \lambda - N(\varepsilon)$, note that for any $i \in C_t$ with maximum degree $\Delta$ in $G[C_t]$, we have
    $$\varepsilon^{-2t} m^{-1/2} \Delta \leq \sum_{j \in N(i)} x_j = \lambda x_i \leq \lambda \cdot \varepsilon^{2t}.$$
    Recalling the standard estimate $\lambda\le \sqrt{2m}$, we see that $G[C_t]$ has max degree at most $\Delta\le \varepsilon^{4t} m^{1/2} \lambda \leq \varepsilon m$, thus it satisfies (b), as desired. 
\end{proof}

\section{Applications}\label{sec-5}

In this section, we present some corollaries of our main results. We focus on spectral analogs of phase transitions of several graph parameters for Nosal graphs. Some of such results were known prior to our work. For example, recent results of Ning and Zhai \cite{NZ2021, NZ2021b} show that every Nosal graph contains $\Omega (\sqrt{m})$ triangles and $\Omega(m^2)$ copies of $C_4$. In 2021, Zhai, Lin and Shu \cite{ZLS2021} 
proved that every Nosal graph contains a large complete bipartite subgraph $K_{2, t}$ with $t = \Omega(\sqrt{m})$. Theorem \ref{thm-booksize} recovers these results up to a constant. 

\subsection{Triangular edges}
A \emph{triangular edge} is an edge contained in a triangle. Extending Mantel's theorem, 
Erd\H{o}s, Faudree and Rousseau \cite{EFR1992} proved that 
every $n$-vertex graph with more than 
$\lfloor n^2/ 4\rfloor$ edges contains at least 
$2\lfloor n/2\rfloor +1$ triangular edges. 
Recently, Li, Feng and Peng \cite{LFP2024-triangular} investigated the spectral problems on triangular edges and showed a spectral version of 
Erd\H{o}s--Faudree--Rousseau's theorem. 
Moreover, they proposed the following conjecture.

\begin{conj}[See \cite{LFP2024-triangular}] 
\label{thm-main2} 
Every $m$-edge Nosal graph has at least $\sqrt{m}$ triangular edges. 
\end{conj}

We remark that if Conjecture \ref{thm-main2} is true, 
then the bound $\sqrt{m}$ is the best possible. 
Similar with the graphs in Example \ref{exam-KST}, 
we consider $G=K_{s,t}^+$, where we take $s=2(\sqrt{m} +1)$ and $t=\frac{1}{2}(\sqrt{m} -1)$. 
We can easily see that $G$ has exactly $2t+1=\sqrt{m}$ triangular edges. 

Theorem \ref{thm-booksize} immediately implies a weak version of Conjecture \ref{thm-main2}.

\begin{cor}
Every $m$-edge Nosal graph contains more than $\frac{1}{72}\sqrt{m}$ 
triangular edges. 
\end{cor}

\subsection{Maximum size of chordal subgraphs}

A \emph{chordal graph} is a graph with no induced cycles of length at least $4$. It is well-known that a graph is chordal if and only if 
it can be constructed from a single-vertex graph by 
iteratively adding a vertex and connecting it to a clique of 
the current graph. This is called a perfect elimination ordering. 
In 1989, 
Erd\H{o}s, Gy\'{a}rf\'{a}s, Ordman and Zalcstein \cite{EGOZ1989} proved that 
if $n$ is even and $G$ is a graph on $n$ vertices with 
$n^2/4 +1$ edges, then $G$ contains a chordal subgraph with at least  $3n/2 -1$ edges. 
This bound is achieved by the graph $K_{n/2,n/2}^+$. 
We refer to \cite{GS2023} for recent progress. 
Observe that a book in a graph is a chordal subgraph.  
From Theorem \ref{thm-booksize}, we can obtain the following result.

\begin{cor}
Every $m$-edge Nosal graph has a chordal subgraph with at least $\frac{1}{72} \sqrt{m}$ edges. 
\end{cor}

The bound is tight up to a constant factor by the graphs in Examples \ref{exam-KST} and \ref{exam-prism}.

\subsection{$K_4$-saturating edges}
An edge in the complement $\overline{G}$ is called 
a \emph{$K_4$-saturating edge} if the addition of this edge to $G$ creates a copy of $K_4$. 
Disproving a conjecture of Erd\H{o}s and Tuza, 
 Balogh and Liu \cite{BL2014} proved that every $n$-vertex $K_4$-free graph 
 with $\lfloor n^2/4\rfloor +1$ edges must have at least 
 $\frac{2}{33}n^2 - \frac{3}{11}n$ $K_4$-saturating edges. 
 Moreover, they constructed a graph with at most 
 $\frac{2}{33}n^2 - \frac{7}{33}n$ $K_4$-saturating edges. 

Now consider an $m$-edge $K_4$-free Nosal graph $G$. Note that any pair of nonadjacent vertices in the large book guaranteed by Theorem \ref{thm-booksize} form a $K_4$-saturating edge. Thus, the same phase transition on the number of $K_4$-saturating edges occurs (jumping from 0 to $\Omega(m)$) for $K_4$-free graphs when their spectral radii go beyond $\sqrt{m}$.

\begin{cor} \label{coro-K4-sat} 
Every $m$-edge $K_4$-free Nosal graph has $\Omega(m)$ $K_4$-saturating edges.
\end{cor}

Note that the linear-in-$m$ bound above is optimal: the $K_4$-free Nosal graphs in Examples \ref{exam-KST} and \ref{exam-prism} contain $O(m)$ $K_4$-saturating edges.

\subsection{Degree power in a graph}

As the last application, we give a bound on degree powers in graphs with forbidden generalized books. Nikiforov and Rousseau \cite[eq. (4)]{NR2004} proved that if $G$ is a $K_{r+1}$-free graph on $n$ vertices with $m$ edges, then 
\begin{equation}   \label{eq-NR}
\sum_{i=1}^n d_i^2 \le 2\left( 1-\frac{1}{r}\right)mn. 
\end{equation}
Moreover, the equality holds if and only if $G$ is a complete bipartite graph for $r=2$, or  $G$ is a regular complete  $r$-partite graph for $r\ge 3$, that is, $r$ divides $n$ and $G=T_{n,r}$. By Cauchy--Schwarz's inequality, we see that (\ref{eq-NR}) implies Tur\'{a}n theorem. See also \cite{BK2012} for work on degree powers.

The original proof of (\ref{eq-NR}) is combinatorial. 
Now we provide two algebraic proofs. 

\begin{proof}[Algebraic proof of (\ref{eq-NR})]
The Hofmeistar inequality (or Rayleigh's formula) gives that $\lambda^2 (G) \ge \frac{1}{n} \sum_{i=1}^n d_i^2,$
with equality if and only if $G$ is either regular or bipartite semi-regular. Recall that if 
$G$ is a $K_{r+1}$-free graph, then $\lambda^2(G) \le 
\left( 1-\frac{1}{r} \right)2m,$
and the equality holds if and only if 
$G$ is a complete bipartite graph for $r=2$, 
or a regular complete  $r$-partite graph for $r\ge 3$.  
Combining these two inequalities, 
we can obtain (\ref{eq-NR}) immediately. 
\end{proof}

We give another algebraic proof of (\ref{eq-NR}). 
We use $q (G)$ to denote the spectral radius of 
the signless Laplacian matrix of $G$. 
  It was proved in  \cite{HJZ2013} that if 
$G$ is a $K_{r+1}$-free graph, then 
$q(G)\le q(T_{n,r})$, with equality if and only if $G=T_{n,r}$. Since $q(T_{n,r})\le (1- \frac{1}{r})2n$, it follows that   
\[   \frac{1}{m} \sum_{i=1}^n d_i^2 \le q(G)  \le \left( 1-\frac{1}{r} \right)2n.\]  
This completes the second algebraic proof. 
Using a similar argument as the first algebraic proof, we can extend 
(\ref{eq-NR}) to $B_{r,k}$-free graphs using Theorem \ref{cor:joints}.

\begin{cor}  \label{coro-deg-square}
For every fixed $r\ge 2$ and $k\ge 1$, 
there exists an integer $m_0$ such that 
if $m\ge m_0$ and $G$ is a $B_{r,k}$-free graph on 
$n$ vertices with $m$ edges, then 
\begin{equation*} 
\sum_{i=1}^n d_i^2 \le 2\left( 1-\frac{1}{r}\right)mn.
\end{equation*} 
\end{cor}

Since $B_{r,k}$ is color-critical\footnote{A graph is called color-critical
if it contains an edge whose deletion reduces its chromatic number.}, a result of Simonovits implies 
$e(G) \le e(T_{n,r}) \le \left( 1- {1}/{r}\right) {n^2}/{2}$. 
Thus, Corollary \ref{coro-deg-square} yields that 
for fixed integers $r\ge 2$ and $k\ge 1$, 
if $G$ is a $B_{r,k}$-free graph with large order and size, 
then  
$ \sum_{i=1}^n d_i^2 \le \left( 1- {1}/{r}\right)^2n^3$. 
This result also extends Tur\'{a}n's theorem.

\section{Concluding remarks}

\label{sec-6}

\noindent\textbf{Books.} We have shown in Theorem \ref{thm-booksize}  that 
$  bk(G)>  \frac{1}{144}\sqrt{m}$ for every $m$-edge Nosal graph $G$. 
On the other hand, Example \ref{exam-prism} shows that there exists a Nosal graph 
$G$ with $bk(G) \le (\frac{1}{3}+o(1))\sqrt{m}$. 
It would be interesting to know what the right constant is.

\begin{problem} \label{conj-bk-13}
Does every $m$-edge Nosal graph contain a book of size $\frac{1}{3}\sqrt{m}$?
\end{problem} 

Recall that 
Edward's theorem states that $bk(G) > n/6$ if $m > n^2/4$. 
Here, we remark that Problem \ref{conj-bk-13}, if true, implies Edward's theorem, since the spectral condition $\lambda (G) > \sqrt{m}$ is weaker than the edge-condition $m > n^2/4$. 
Indeed, if $G$ is an $n$-vertex graph  with $m > n^2/4$ edges, then $\lambda (G)\ge 2m/n > n/2$ and so $\lambda^2 (G) \ge (2m/n)^2 > m$. 

\medskip

\noindent\textbf{Even cycles \&  Complete bipartite graphs.} 
Theorem \ref{thm:lim} determines the asymptotic counts of 4-cycles in Nosal graphs.   
A natural direction to explore is to extend the spectral supersaturation result of $C_4$ to the even cycle $C_{2k}$, the complete bipartite graph $K_{t,t}$ and in general other bipartite graphs (and establish a spectral version of the Sidorenko conjecture). 

\medskip

\noindent\textbf{Color-critical graphs.} 
Another interesting direction to explore is to understand the spectral supersaturation for color-critical graphs with chromatic number $r+1$ in an $m$-edge graph with 
spectral radius large than $\sqrt{(1- {1}/{r})2m}$. 
As an example, we see from 
\Cref{cor-extend-triangle} that if $r\ge 2$ and 
$\lambda (G)> \sqrt{(1-1/r)2m}$, then $G$ contains $\Omega_r(m^{\frac{r-1}{2}})$ copies of $K_{r+1}$.

Here is another example. Let $C_4^+$ be the \emph{kite} graph obtained from $C_4$ by adding a chord. Note that kite is color-critical and $C_4^+ = B_2=K_4^- $. Mubayi \cite{Mubayi2010} proved that 
every $n$-vertex graph with more than $\lfloor n^2/4\rfloor$ edges contains at least ${\lfloor n/2 \rfloor \choose 2}$ copies of $C_4^+$. For an $m$-edge Nosal graph $G$, by Theorem \ref{thm-booksize}, we know that 
$G$ contains a subgraph $B_k$  
with $k= \Omega(\sqrt{m})$. As we get a kite by taking any two triangles in this book, we obtain the following result. 

\begin{cor} \label{coro-C4+}
Every $m$-edge Nosal graph contains $\Omega(m)$ copies of $C_4^+$. 
\end{cor}

The  bound in Corollary \ref{coro-C4+} is tight up to a constant factor, 
since the graph $H$ in Example \ref{exam-KST} 
has ${t \choose 2}\approx \frac{1}{8} m$ copies of $C_4^+$.   
Observe that $C_4^+$ contains both $C_3$ 
and $C_4$ as subgraphs. 
Also note the difference between the supersaturation of $C_4^+$ in Nosal graphs and that of $C_3$ and $C_4$: every Nosal graph contains $\Theta (\sqrt{m})$ copies of $C_3$, $\Theta(m^2)$ copies of $C_4$, and $\Theta(m)$ copies of $C_4^+$.

\medskip

We summarize the discussions above in the following table. 

\begin{table}[H]
\centering
\begin{tabular}{ccccccccccc}
\toprule
  & $e(G)> \lfloor n^2/4\rfloor$ & $\lambda (G) > \sqrt{m}$ \\
\midrule
 $\#K_3$ & $\lfloor \frac{n}{2} \rfloor$    \cite{Erd1955,Erdos1964}  
 &  $\lfloor \frac{1}{2}(\sqrt{m} -1) \rfloor$ 
 \cite{NZ2021b}  \\
 $\#C_4$ & $\Omega (n^4) $  & $(\frac{1}{8} - o(1))m^2 $  \\ 
 Booksize $bk(G)$ & $\lceil \frac{n}{6} \rceil $  \cite{KN1979} 
 &  $\Omega (\sqrt{m})$  \\
      Largest $K_{2,t}$ & $\Omega(n)$ & $\Omega (\sqrt{m})$  \\ 
 $\#$Triangular edges & $2\lfloor \frac{n}{2}\rfloor +1$  \cite{EFR1992} & $\Omega (\sqrt{m})$ \\ 
 Maximal chordal subgraph & $\frac{3}{2}n-1$  \cite{EGOZ1989}  & $\Omega (\sqrt{m})$ \\  
 $\#K_4$-saturating edges & $\frac{2}{33}n^2 - \Theta(n)$  \cite{BL2014} 
 &  $\Omega (m)$ \\  
  $\#C_4^+$ & ${\lfloor n/2\rfloor \choose 2}$ 
 \cite{Mubayi2010} &  $\Omega(m)$ \\ 
\bottomrule 
\end{tabular}
\caption{Some spectral supersaturation results.}
 \label{tab-SSSR}
\end{table}

We can observe from Table \ref{tab-SSSR} that there exists a correspondence between $n$  and $\sqrt{m}$ in the vertex- and edge-spectral conditions. 
There have been significant effort devoted to determining the optimal constant factors on the left column. It would be interesting to determine the asymptotic spectral counts for the right column of Table \ref{tab-SSSR} as well. 

We plan to return to these topics in the near future.

\section*{Acknowledgements}
We would like to thank 
Yuval Wigderson for the reference \cite{Erd1969}. This work was initiated at the 1st IBS ECOPRO student research program in fall 2023. The first and third authors would like to thank ECOPRO group for hosting them.

\frenchspacing


\end{document}